\documentclass[11pt]{article}
\usepackage{amssymb,amsfonts,amsmath,latexsym,epsf,tikz,url,graphicx}
\newtheorem{theorem}{Theorem}[section]

\newtheorem{corollary}[theorem]{Corollary}

\newtheorem{rem}[theorem]{Remark}

\newcommand{\qed}{\hfill $\square$\medskip}

\begin{document}
\title{$2$-Restricted Optimal Pebbling Number of Some Graphs}

 %\author{Ali Delavar Khalafi$^{}$\footnote{Corresponding author}
\author{J.G. Dehqan \and S. Alikhani$^{}$\footnote{Corresponding author} \and A. Delavar Khalafi     \and F. Aghaei}

\date{\today}

\maketitle

\begin{center}

   Department of Mathematical Sciences, Yazd University, 89195-741, Yazd, Iran\\
\medskip

  {\tt dehqanm@stu.yazd.ac.ir}\\
  {\tt alikhani@yazd.ac.ir}\\
   {\tt delavarkh@yazd.ac.ir}\\
{\tt aghaeefatemeh29@gmail.com}\\

\end{center}

\begin{abstract}
 Let $G=(V,E)$ be a simple graph. A pebbling configuration on $G$ is a function $f:V\rightarrow \mathbb{N}\cup \{0\}$ that assigns a non-negative integer number of pebbles to each vertex. The weight of a configuration $f$ is $w(f)=\sum_{u\in V}f(u)$,  the total number of pebbles.
  A pebbling move consists of removing two pebbles from a vertex $u$ and placing one pebble on an adjacent vertex $v$. A configuration $f$ is a $t$-restricted pebbling configuration ($t$RPC) if no vertex has more than $t$ pebbles. The $t$-restricted optimal pebbling number $\pi_t^*(G)$ is the minimum weight of a $t$RPC on $G$ that allows any vertex to be reached by a sequence of pebbling moves. 
 The distinguishing number $D(G)$ is the minimum number of colors needed to label the vertices of $G$ such that the only automorphism preserving the coloring is the trivial one (i.e., the identity map).  
 In this paper, we investigate the $2$-restricted optimal pebbling number of 
 trees $T$ with $D(T)=2$ and radius at most $2$ and enumerate their $2$-restricted optimal pebbling configurations.
 Also we study the $2$-restricted optimal pebbling number of some graphs that are of importance in chemistry such as some alkanes.  
\end{abstract}


\medskip
\noindent{\bf AMS Subj.\ Class.}: 05C30, 05C69.  

\section{Introduction and Definitions}
%%%%%%%%%%%%%%%%%%%%%%%%%%%%%%%%%%%%%%%%%%%%%%%%%%%%%%%%%%%%%%%%%%%%%%%%%%%%%%%%%
%%%%%%%%%%%%%%%%%%%%%%%%%%%%%%%%%%%%%%%%%%%%%%%%%%%%%%%%%%%%%%%%%%%%%%%%%%%%%%%%%

Graph pebbling is like a number of network models, including network flow, transportation,
and supply chain, in that one must move some commodity from a set of sources to a set
of sinks optimally according to certain constraints. Network flow constraints restrict flow
along edges and conserve flow through vertices, and the goal is to maximize the amount
of commodity reaching the sinks. The transportation model includes per unit costs along
edges and aims to minimize the total cost of shipments that satisfy the source supplies and
sink demands. At its simplest, the supply chain model ignores transportation costs while
seeking to satisfy demands with minimum inventory. The graph pebbling model introduced
by Chung \cite{Chung} also tries to meet demands with minimum inventory, but constrains movement
across an edge by the loss of the commodity itself, much like an oil tanker using up the fuel
it transports.

Let $G = (V, E)$ be a simple graph. Let
$ V \rightarrow  \mathbb{N} \cup \{ 0 \}$ 
be a function that assigns to each vertex $u \in V$ a nonnegative integer $f(u)$. We  say that $u$ has been assigned $f(u)$ pebbles. 
Let $w (f) = \sum_{u \in V}{f(u)}$ equal the total number of pebbles assigned by the function $f$ and that $f$ is a pebbling configuration. 
A pebbling move consists of removing two pebbles from a vertex
$u \in V $ and then adding one pebble to an adjacent vertex $v \in N(u)$. A pebbling configuration $f$ is said to be solvable if for
every vertex $v$, there exists a sequence (possibly empty) of pebbling moves that results in a pebble on $v$.
The following definition appears in many papers in graph pebbling. The pebbling number $\pi(G)$ equals the minimum number
 $k$ such that every pebbling configuration $ V \rightarrow  \mathbb{N} \cup \{ 0 \}$  with $w(f) = k $ is solvable. Thus, the central focus of graph
pebbling is to determine a minimum number of pebbles so that no matter how they are placed on the vertices of a graph $G$, 
there will always be a sequence of pebbling moves that can move at least one pebble to any specified vertex of a graph $G$.

 The concept of pebbling was introduced in the literature by Chung in \cite{Chung}, where she proved that the pebbling number of the $n$-cube equals $2^n$. This result was used to give an alternate proof of a number theoretic theorem of Lemke and Kleitman \cite{Lemke}. Other applications of graph pebbling might include transportation of material. For example, in percolation theory in physics one considers pouring a liquid through a porous substrate. In the process of doing this some of the liquid is absorbed by the substrate. One then considers the probability that some liquid will flow all the way through the substrate. The loss of liquid would correspond to discarding one of the two pebbles in a pebbling move; while the liquid flowing through could be measured by covering the vertices. For another example, if a pebbling move is viewed as a transportation problem, one desires to move a unit of some object from a vertex $u$ to an adjacent vertex $v$ with a transportation cost of one unit, e.g. a gallon of gas.  

Pachtor et al. \cite{Pachtor} defined the optimal pebbling number $\pi^{*}(G)$ to be the minimum weight of a solvable pebbling
configuration of G. A solvable pebbling configuration of $G$ with weight $\pi^{*}(G)$ is called a $\pi^{*}$-configuration. Optimal pebbling
was studied further in \cite{Bunde,Friedman,Fu1,Fu2,Herscovici1,Herscovici2, Moews,Muntz,Shiue,Ye}. The decision problem associated with computing the optimal pebbling number was shown to be NP-Complete in \cite{Milans}.

in this paper, we consider a generalization of the optimal pebbling number.
 We say that a pebbling configuration $f$ is a $t$-restricted pebbling configuration (abbreviated tRPC) if $f(u) \le t$ for all  $u \in V$.
 We define the $t$-restricted optimal pebbling number $\pi_t ^*(G)$ to be the minimum weight of a solvable tRPC on $G$. If $f$ is a solvable tRPC on $G$ with $w(f) = \pi ^*(G)$, then $f$ is called a $\pi_t ^*$-configuration of $G$.

In graph theory, a {\em dominating set} of a graph $G$ is a subset $S$ of the vertex set $V$ such that every vertex in $V\setminus S$ is 
adjacent to at least one vertex in $S$. The minimum size of dominating sets of $G$ is called the domination number of $G$ and is denoted
 by $\gamma(G)$. 
Dominating sets and domination numbers are well-studied subjects in graph theory and to delve into the topic of domination in graphs,
 readers can refer to books such as \cite{10,11,11',12}. Roman domination number is denoted by $\gamma_R(G)$ and by its definition
 we see that for any graph $G$, $\pi_2^*(G)\leq \gamma_R(G)$ and since
 $\gamma_R(G)\leq 2\gamma(G)$, so $\pi_2^*(G)\leq 2\gamma(G)$.  For the path $P_n$, 
$\pi_2^*(P_n)=\pi^*(P_n)=\gamma_R(P_n)=\lceil\frac{2n}{3}\rceil.$
It is interesting that to characterize graphs $G$ with small number $\pi_2^*(G)$ we need to consider the domination and total domination
 number of $G$. To see the characterization of graphs $G$ with $\pi_2^*(G)=2,3,4, 5$ refer to  \cite{Alikhani,Chellali}.
 In \cite{Chellali} proved that if $T$ is a tree of order $n \ge 3$ with $l$ leaves, then $\pi_2^*(T) \le n-l+1$, and this bound is sharp. 
 Also it has proved that for any nontrivial tree $T$ of order $n$,  $\pi_2^*(T) \le \lceil \frac{5n}{7} \rceil$.

Distinguishing labeling was first defined by Albertson and Collins \cite{Albertson} for graphs. 
A labeling of a graph $G$, $\phi : V (G ) \rightarrow \{1,2,\cdots,r \}$, is said to be $r$-distinguishing if no nontrivial automorphism of $G$ 
preserves all the vertex labels.
The {\em distinguishing number} $D(G)$ of a graph $G$ is the smallest $r$ such that $G$ admits a distinguishing $r$-labeling. 

Alkanes are important raw materials of the chemical industry and the principal constituent of gasoline and lubricating oils.
 For example Natural gas mainly contains methane and ethane and is used for heating and cooking purposes and for power utilities
 (gas turbines). An alkane consists of hydrogen and carbon atoms arranged in a tree structure in which all the carbon–carbon bonds
 are single. Alkanes have the general chemical formula $C_nH_{2n+2}$.

 \medskip 
 In the next section, we investigate the $2$-restricted optimal pebbling number of 
 trees $T$ with $D(T)=2$ and radius at most $2$ and enumerate their $2$-restricted optimal pebbling configurations
 We study the $2$-restricted optimal pebbling number of some graphs that are of importance in chemistry such as
 some alkanes and dendrimers in Section 3.

\section{Results for trees $T$ with $D(T)=2$}
 In this section we compute the $2$-restricted optimal pebbling number of trees whose distinguishing  number is $2$ and radius at most is $2$.  

\begin{figure}[h!]
\begin{center}
\includegraphics[scale=0.6]{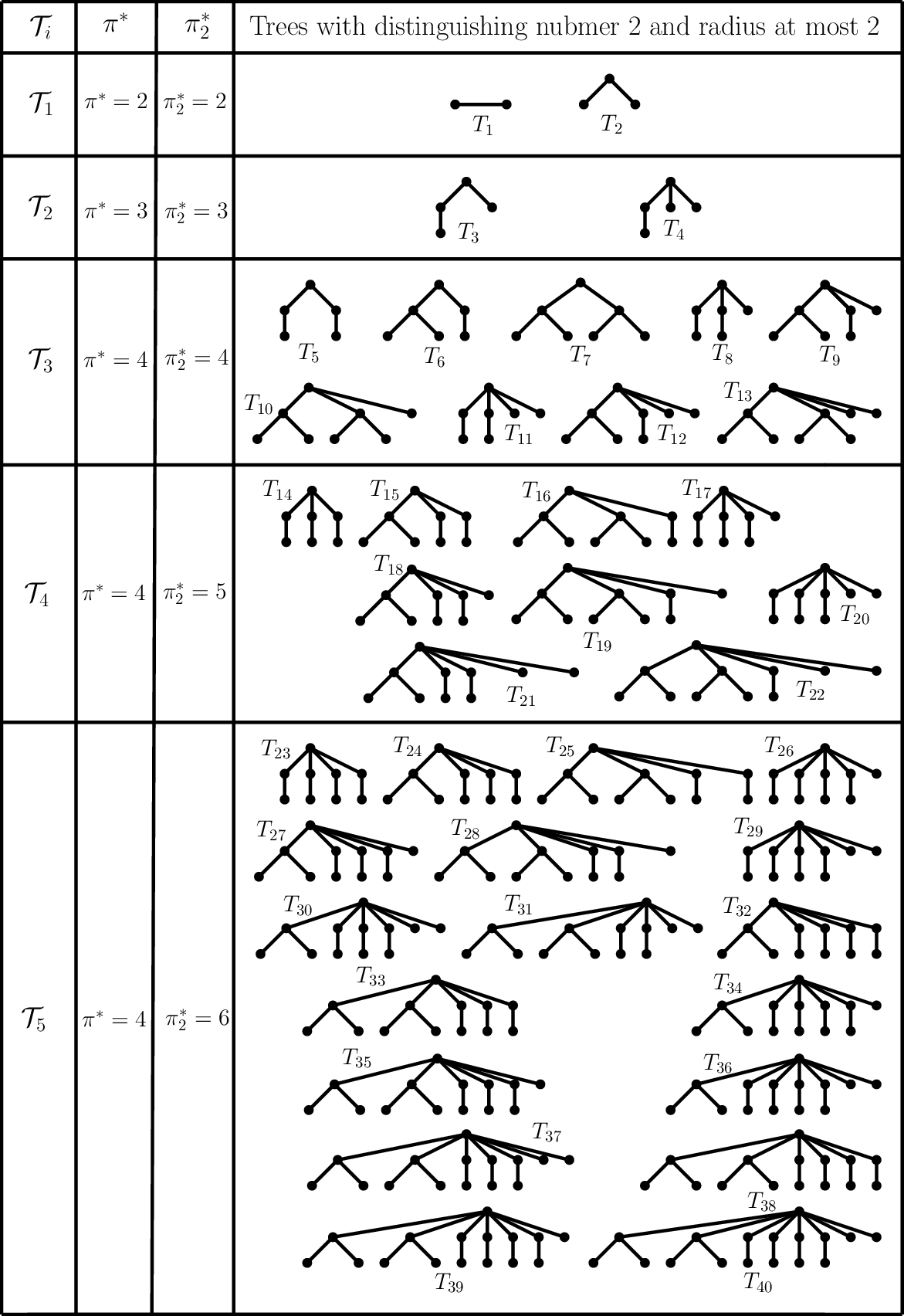}
\caption {2-restricted optimal pebbling number.}
 \label{2}
 \end{center}
\end{figure} 
Let $\mathcal{T}=\bigcup_{i=1}^6\mathcal{T}_i$ denote the family of trees $T$ with $D(T)=2$ and  radius at most $2$ (see Figure \ref{2}).

 \begin{theorem}
  If $T$ is a tree in $\mathcal{T}_i $ for $i=1,2,3$, then $ \pi_2 ^*(T)=\pi^*(T)$.
  \end{theorem} 
  \begin{proof}  
  	Since $\pi^*(T)\leqslant\pi_2 ^*(T)$ for  every graph $T$, so by placing two pebbles on the root and a pebble on the 
  vertices in the level one that have children, we have $2$-restricted optimal pebbling configuration with weight total number  $\pi^*(T)$.
  \qed
  \end{proof}
 
  \begin{theorem}
   $\pi_2^*(T)=5$, for the trees  $T \in \mathcal{T}_4$.
\end{theorem}
 \begin{proof}
 The trees in the family $\mathcal{T}_4$ are constructed by adding a path with two edges to the root of each tree in the family 
 $\mathcal{T}_3$. Since the root in any tree of family $\mathcal{T}_3$ can receive at most $2$ pebbles, to solve the new path vertices,
 we need at least one more pebble. Therefore $\pi_2^*(T)\geqslant 5$. Now, we consider the $2$-restricted optimal pebbling 
 configuration with $5$ pebbles, where $2$ pebbles are placed on the root and a pebble on the first-level vertices that have children.
So the result is obtained.
 \end{proof}
 
 \begin{theorem}
 $\pi_2^*(T)=6$,  for the trees  $T \in \mathcal{T}_5$.
\end{theorem}
\begin{proof}
Since trees in family $\mathcal{T}_5$ contain a subgraph of trees in family $\mathcal{T}_4$, and any $2$-restricted optimal pebbling 
configuration of tree  $T\in\mathcal{T}_4$ can reach at most $2$ pebbles on the root, therefore,  $\pi_2^*(T)\geqslant 6$. Now we 
put two pebbles on the root and  two pebbles on two vertices of the first-level. Therefore the result is obtained.
 \qed
 \end{proof}

 \begin{corollary}
If $T\in \mathcal{T}$, then $2\leq \pi_2^*(T)\leq 6$ and so $D(T)\leq \pi_2^*(T)$.
\end{corollary} 

The study of the number of $2$-restricted optimal pebbling configurations of a graph $G$ is an interesting and a natural problem. 
 Let denote the number of the $2$-restricted optimal pebbling configurations of graph $G$ with $P_2^*$.  
 Here, we classify the trees  in the family $\mathcal{T}$  according the number of $2$-restricted optimal pebbling configurations.

   \begin{theorem}\label{np}
 	If $P_2^*$ is the number of the $2$-restricted optimal pebbling configurations of graph $G$, then for tree $T_i\in \mathcal{T}$ 
	(Figure \ref{2}) ($1\leq i\leq 40$) we have
 \begin{enumerate}
 \item[(i)]
 The only tree in the family $\mathcal{T}$ with unique $2$-restricted optimal pebbling configuration is the tree $T_2$ which is a star graph. In other word, $P_2^*(T_2)=1$.
 \item[(ii)] 
For $T\in \{T_1,T_4\}$, $P_2^*(T)=3$, and for  $T\in \{T_3,T_7,T_{10},T_{13}\}$,  $P_2^*(T)=4$. 
\item[(iii)]
For $T\in \{T_{9},T_{12}\}$, $P_2^*(T)=6$,  $P_2^*(T_6)=8$, for $T\in \{T_8,T_{11}\}$, $P_2^*(T)=9$, and for $T\in \{T_{32},T_{33}\}$,  $P_2^*(T)=10$.
\item[(iv)]
For $T\in \{T_{16},T_{19},T_{22}\}$,  $P_2^*(T)=11$, $P_2^*(T_5)=13$, for $T\in \{T_{34},T_{35},T_{38}\}$, $P_2^*(T)=15$,
\item[(v)]
 For $T\in \{T_{15},T_{18},T_{21}\}$ , $P_2^*(T)=17$,
 for  $T\in \{T_{36},T_{37},T_{39}\}$ , $P_2^*(T)=21$, and for $T\in \{T_{14},T_{17},T_{20}\}$, $P_2^*(T)=26$, 

\item[(vi)] 
$P_2^*(T_{40})=28$, $P_2^*(T_{25})=35$, $P_2^*(T_{28})=39$.

\item[(vii)] 
$P_2^*(T_{31})=44$, $P_2^*(T_{24})=52$
$P_2^*(T_{27})=56$, $P_2^*(T_{30})=61$.

\item[(viii)] 
$P_2^*(T_{23})=78$, $P_2^*(T_{26})=82$
$P_2^*(T_{29})=87$.
 \end{enumerate}
 \end{theorem}

 From the results in Theorem \ref{np} we have the following corollary: 
 \begin{corollary} 
 If $T\in \mathcal{T}$, then $1 \leq P_2^*(T)\leq 87$. 
\end{corollary}

Some of $2$-restricted optimal pebbling configurations of some trees have shown in Figure \ref{12}.

\begin{figure}[h!]
\begin{center}
\includegraphics[scale=0.5]{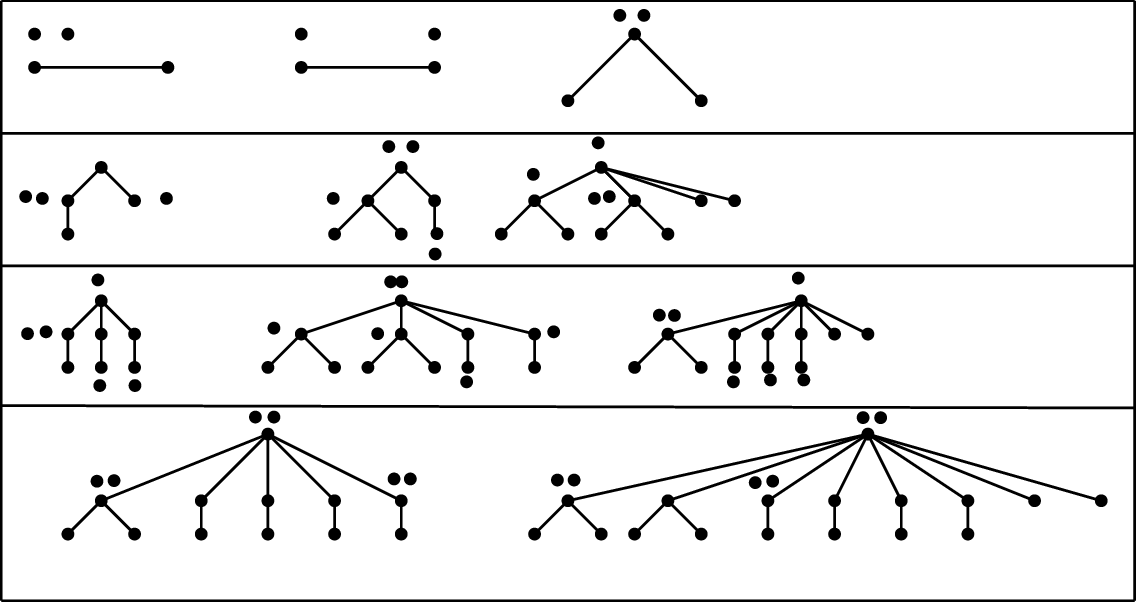}
\caption{The $2$-restricted optimal pebbling configurations of some trees $T\in \mathcal{T}$.}
\label{12}
 \end{center}
  \end{figure}

  \begin{theorem}
 Let $G$ be a connected graph of order  $n\geqslant 2$ and $f$ be a $2$-restricted optimal pebbling configuration. If $f(u)=1$  for some
 $u\in V(G)$, then there is at least two distinct  solvable 2RPC on the graph $G$.
\end{theorem}
\begin{proof}
  We have two distinct 2RPC for the connected graph of order two (fig \ref{12}).  Now let $f$ be a 2RPC on the connected graph $G$ of order $n>2$.
  Since $\pi^*_2\leqslant \gamma_R\leqslant 2\gamma\leqslant 2(\frac{n}{2})$, so $f(v)=2$ for some $v\in V(G)$. If  for some $u\in N(v)$, 
  $f(u)=1$. then by changing  $f(v)=1$ and $f(u)=2$, we have the new 2RPC on $G$. Otherwise there exists a vertex $w\in N(u)$ that  $f(u)=1$ and 
  $f(w)=0$, and so  by changing  $f(u)=0$ and $f(w)=1$, we have the new 2RPC on $G$.
\qed
\end{proof}
\begin{corollary}
If $\pi^*_2(G)=2k+1$ for $k\in \mathbb{N}$, then there is at least two distinct  solvable 2RPC on the graph $G$.
\end{corollary}
  
  \begin{rem}
There exists unique 2RPC  for the path graph $P_{3k}$ with total weight $2k$.  
  \end{rem}
  
  \begin{figure}[h!]
\begin{center}
\includegraphics[scale=1]{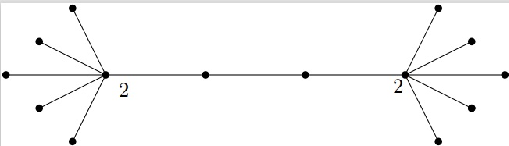}
\caption{The unique $2$-restricted optimal pebbling configuration of graph $G$  }
\label{unique}
 \end{center}
  \end{figure}

     \section{Results for some alkanes } 
    In this section we study the $2$-restricted optimal pebbling number and configurations of some alkanes.

    Alkanes are important raw materials of the chemical industry and the principal constituent of gasoline and lubricating oils. For example Natural gas mainly contains methane and ethane and is used for heating and cooking purposes and for power utilities (gas turbines).
    An alkane consists of hydrogen and carbon atoms arranged in a tree structure in which all the carbon–carbon bonds are single. Alkanes have the general chemical formula 
    $C_n H_{2n+2}$.
    
    In this section, we find the $2$-restricted optimal pebbling number of some Alkanes.
        Methane $(CH_4)$ is the simplest member of the alkane family and the simplest organic compound. Methane is a compound that has a tetrahedral structure and is formed by the
    bonding of four hydrogen atoms and one carbon atom. 
    Ethane is a chemical compound with the formula $C_2H_6$, a member of the hydrocarbon group alkane, and has two-carbon.
    From the $2$-restricted optimal pebbling configuration of $CH_4$ and $C_2H_6$ which have shown in Figures \ref{13} and \ref{14}, we have the following theorem for the $2$-restricted optimal pebbling number 
    of methane and ethane. 
    
   \begin{theorem}
   	\begin{enumerate}
   		 \item[(i)] $\pi_2^*(CH_4)=2$ and $P_2^*(CH_4)=1$. 
   		 
   		 \item[(ii)] $\pi_2^*(C_2H_6)=3$ and $P_2^*(C_2H_6)=2$. 
   		   	\end{enumerate}
        \end{theorem} 
    
    \begin{proof}
	\begin{enumerate}
	\item[(i)]The molecule graph methane is star graph $K_{1,4}$, so $\pi^*_2 (CH_4)=2$ and $P_2^*(CH_4)=1$. 
	\item[(ii)] By considering subgraph $P_4$ of molecule graph ethane $C_2H_6$, we have $\pi^*_2(P_4)\leqslant\pi^*_2(C_2H_6) $. 
	Since $\pi^*_2(T)\leqslant n-l+1$ for any tree, so the result is obtained.
	\qed
	\end{enumerate}
    \end{proof}
    
    \begin{figure}[h!]
    	\begin{center}
    		\includegraphics[scale=0.5]{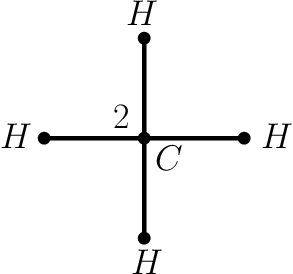}
    		\caption{Unique $2$-restricted optimal pebbling configuration of $CH_4$}
    		\label{13}
    	\end{center}
    \end{figure}
    
    \begin{figure}[ht!]
    	\begin{center}
    		\includegraphics[scale=0.5]{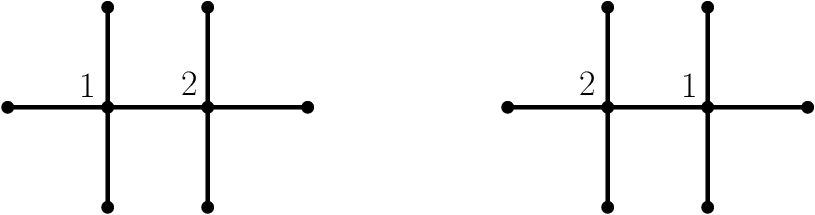}
    		\caption{Two 2-restricted optimal pebbling configurations of $C_2H_6$}
    		\label{14}
    	\end{center}
    \end{figure}
    Propane is a three-carbon alkane with the molecular formula $C_3H_8$.
    Propane is a natural gas (natural gas is about 90 percent methane, 5 percent propane, 5 percent other gases) and is a by-product of natural gas processing and crude oil refining, which is converted into a liquid under pressure.
    
    Butane is an organic compound with the formula $C_4H_{10}$. Butane is a saturated hydrocarbon containing 4 carbons, with an unbranched structure. Butane is primarily used as a gasoline mixture, either alone or in a propane mixture. It is also used as a feedstock for ethylene and butadiene production.

       \begin{figure}[h!]
       	\begin{center}
       		\includegraphics[scale=0.40]{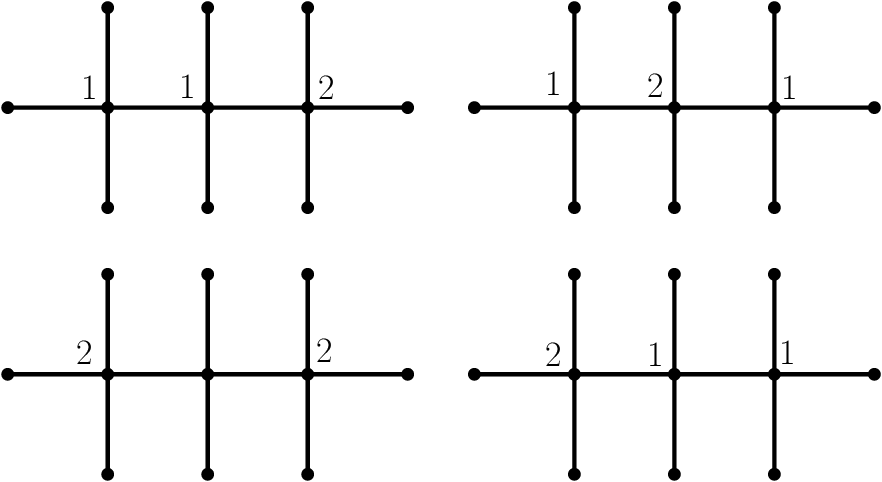}
       		\caption{Four 2-restricted optimal pebbling configurations of $C_3H_8$}
       		\label{15}
       	\end{center}
       \end{figure}
       
       \begin{figure}[h!]
       	\begin{center}
       		\includegraphics[scale=0.40]{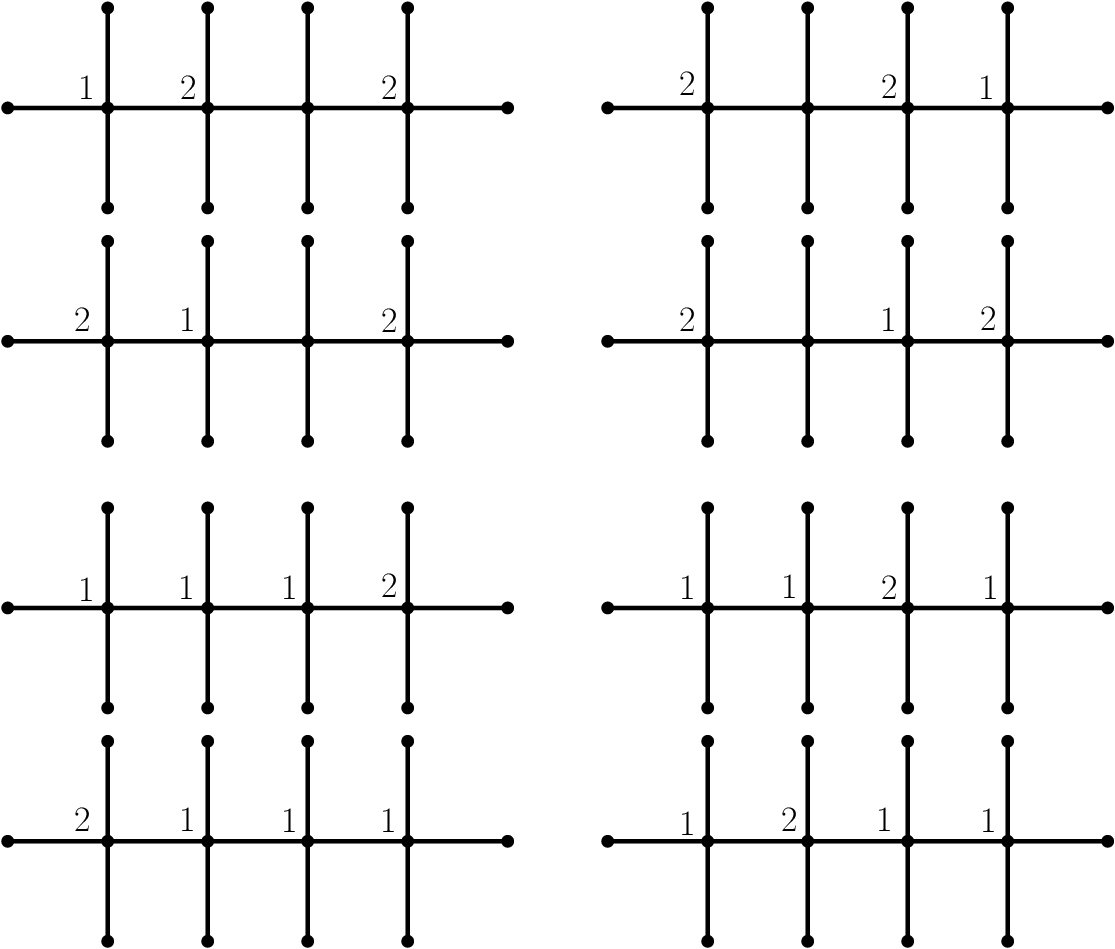}
       		\caption{Eight 2-restricted optimal pebbling configurations of $C_4H_{10}$}
       		\label{16}
       	\end{center}
       \end{figure}
    
     From the $2$-restricted optimal pebbling configuration of $C_3H_8$ and $C_4H_{10}$ which have shown in Figures \ref{15} and \ref{16}, we have the following theorem for the $2$-restricted optimal pebbling number 
     of  propane and butane and for the number of their $2$-restricted optimal pebbling configurations.  
     
     \begin{theorem}
     	\begin{enumerate}
     		\item[(i)] $\pi_2^*(C_3H_8)=4$ and $P_2^*(C_3H_8)=4$. 
     		
     		\item[(ii)] $\pi_2^*(C_4H_{10})=5$ and $P_2^*(C_4H_{10})=8$. 
     	\end{enumerate}
     \end{theorem} 
     
     \begin{proof}
	\begin{enumerate}
	\item[(i)] The molecule graph $C_3H_8 $ is constructed by adding a claw with three edges to the leaf of  graph $C_2H_6$. 
	Since the leaves in the graph $C_2H_6$ can receive at most 1 pebble. To solve the new claw vertices, we need at least one more pebble.
	By considering the 2-restricted optimal pebbling configurations of $C_3H_8$ in Fig \ref{15}, the result is followed.
	\item[(ii)] The molecule graph $C_4H_{10} $ is constructed by adding a claw with three edges to the leaf of  graph $C_3H_8$. 
	The rest of the proof is as before.
	
	\end{enumerate}
	\qed
    \end{proof}

    Pentane is an organic compound with the formula $C_5H_{12}$. It is classified as an alkane with five carbon atoms. Normal pentane is used as a nonpolar solvent in the laboratory and in industry as a reagent in the production of polystyrene foam. Pentane is also used in geothermal power units.
    
    Isobutane, also known as i-butane, 2-methylpropane or methylpropane, is a chemical compound with the molecular formula $C_4H_{10}$ and is an isomer of butane which we denote it by $IB$. Isobutane is the simplest alkane with a tertiary carbon bond, and is used as a feedstock in the petrochemical industry. Over the past decade, concerns about ozone depletion by freons have led to an increase in the use of isobutane as a refrigerant gas, particularly in household refrigerators and freezers.

    \begin{figure}[h!]
    	\begin{center}
    		\includegraphics[scale=0.35]{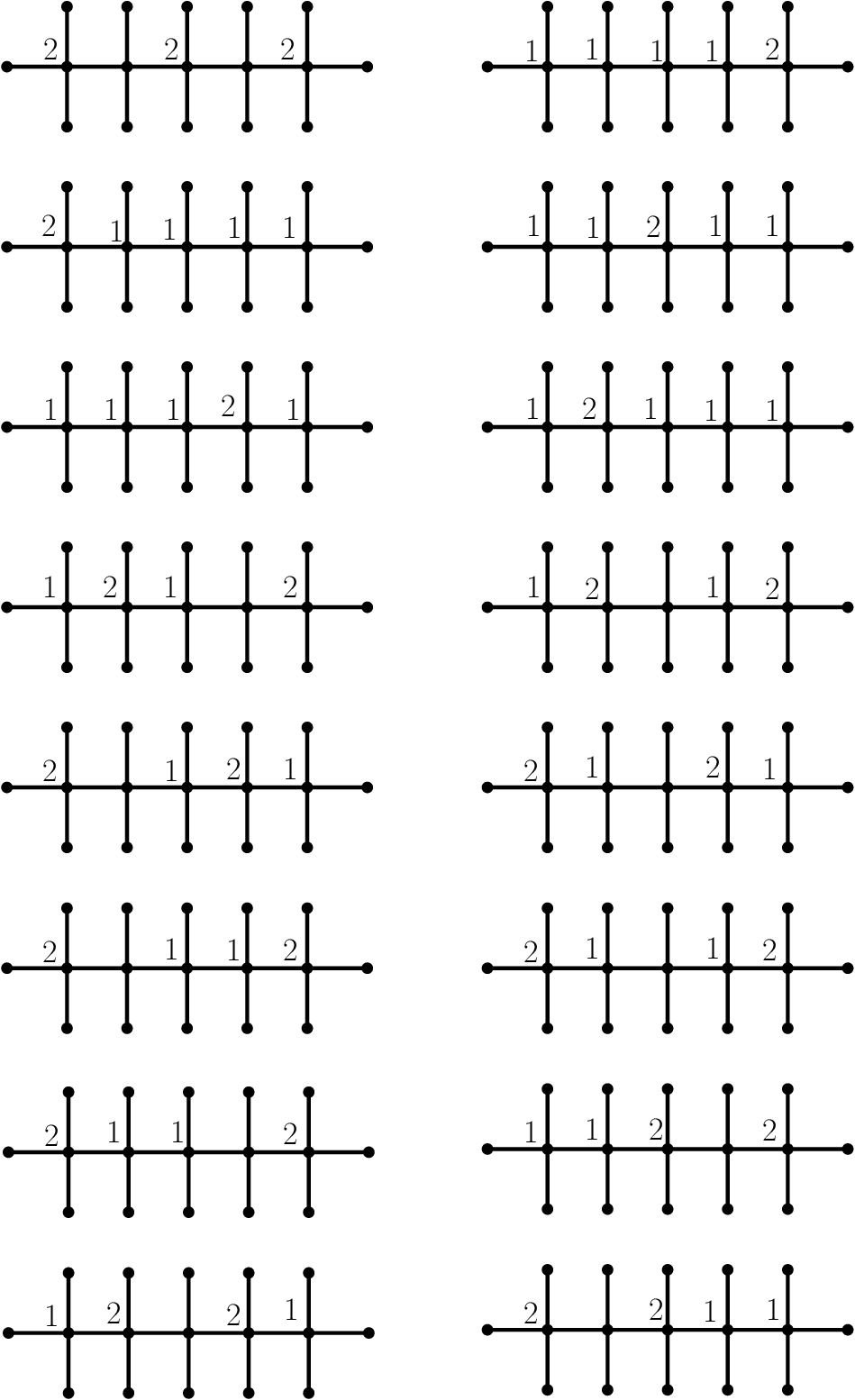}
    		\caption{The 2-restricted optimal pebbling configurations of $C_5H_{12}$}
    		\label{17}
    	\end{center}
    \end{figure}
    
    \begin{figure}[h!]
    	\begin{center}
    		\includegraphics[scale=0.30]{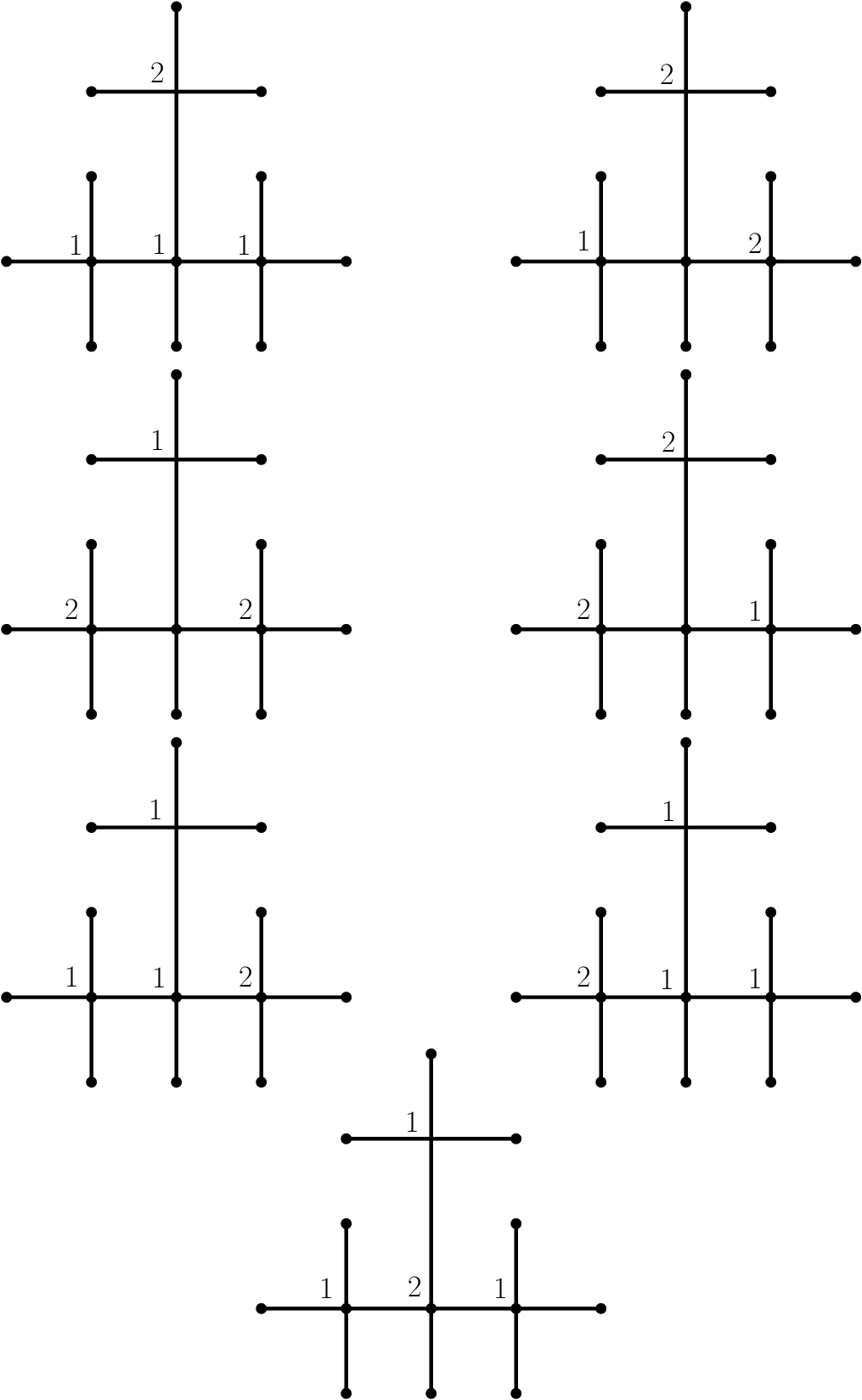}
    		\caption{seven 2-restricted optimal pebbling configurations of 2-methyl-propane}
    		% $C_4H_{10}$
    		\label{18}
    	\end{center}
    \end{figure}

    From the $2$-restricted optimal pebbling configuration of $C_5H_{12}$ and $IB$ which have shown in Figures \ref{17} and \ref{18}, we have the following theorem for the $2$-restricted optimal pebbling number 
    of  pentane and isobutane and for the number of their $2$-restricted optimal pebbling configurations.

    \begin{figure}[h!]
    	\begin{center}
    		\includegraphics[scale=0.30]{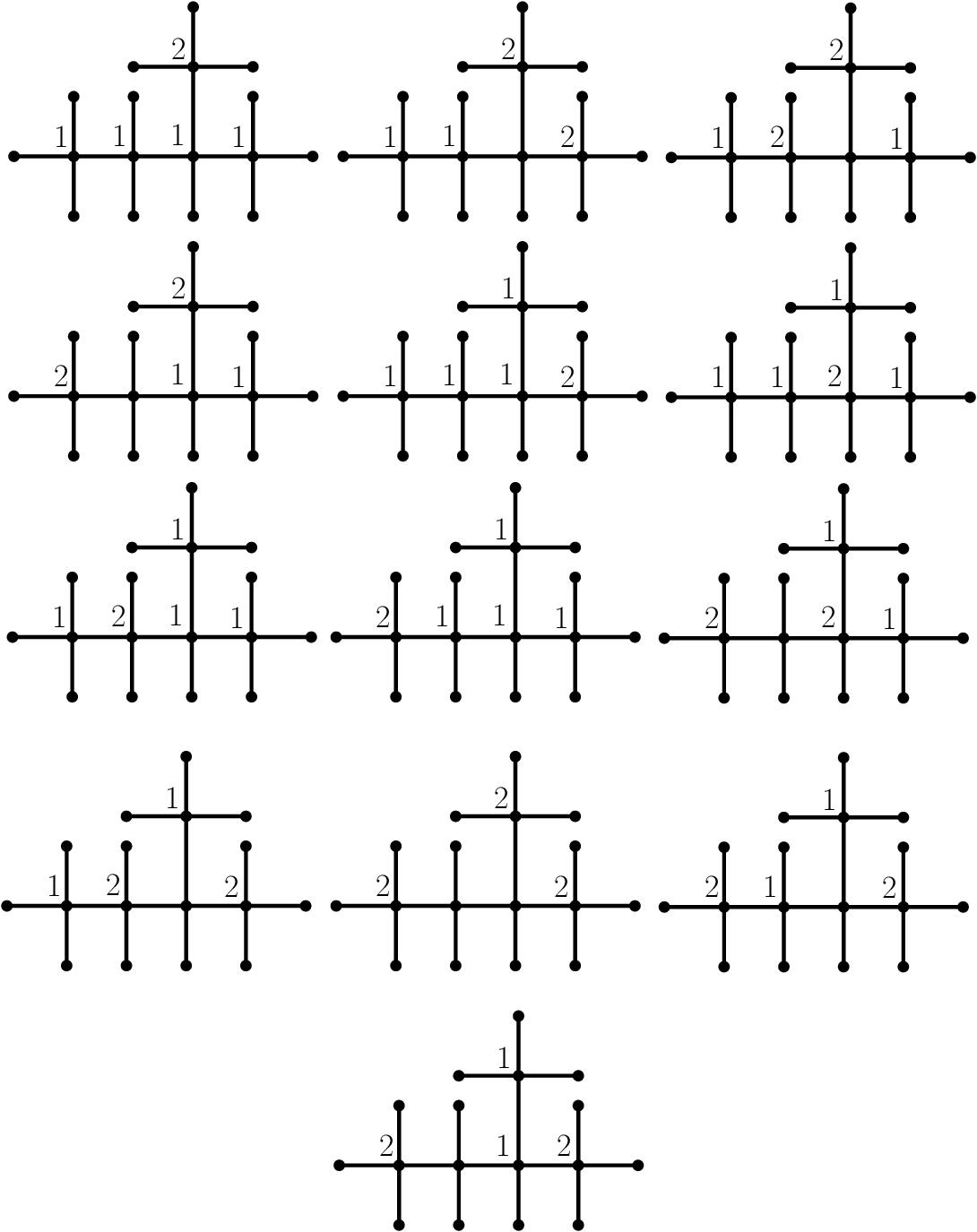}
    		\caption{Thirteen $2$-restricted optimal pebbling configurations of 2-methyl-butane}	\label{19}
    		%$C_5H{12}$
    	\end{center}
    \end{figure}
    
     \begin{figure}[h!]
     	\begin{center}
     		\includegraphics[scale=0.40]{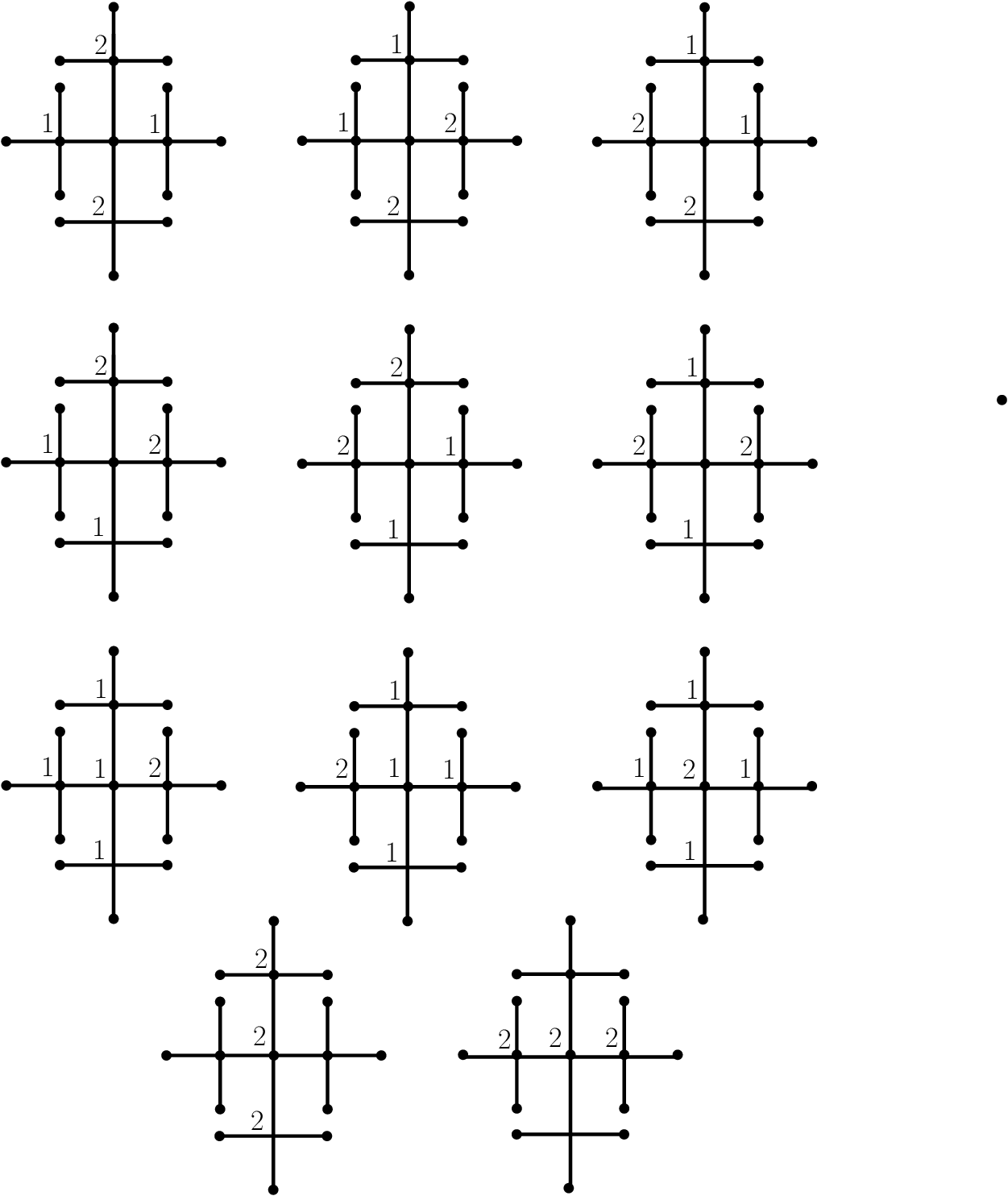}
     		\caption{The $2$-restricted optimal pebbling configurations of Neopentane}
     		\label{20}
     	\end{center}
     \end{figure}
    
    \begin{theorem}
    	\begin{enumerate}
   	\item[(i)] $\pi_2^*(C_5H_{12})=6$ and $P_2^*(C_5H_{12})=16$. 
    		
   \item[(ii)] $\pi_2^*(IB)=5$ and $P_2^*(IB)=7$. 
    	\end{enumerate}
    \end{theorem} 
    
     \begin{proof}
	\begin{enumerate}
	\item[(i)] The molecule graph $C_5H_{12} $ is constructed by adding a claw with three edges to the leaf of  graph $C_4H_{10}$. 
	
	\item[(ii)] The molecule graph $IB $ is constructed by adding a claw with three edges to the leaf of  graph $C_3H_8 $. 
	
	\end{enumerate}
	\qed
    \end{proof}

    Isopentane, also called methylbutane or 2-methylbutane, is a branched-chain saturated hydrocarbon (an alkane) with five carbon atoms, with formula $C_5H_{12}$. Let to denote it by $IP$.  
    The $2$-restricted optimal pebbling number of $IP$ is 6. This alkane has thirteen $2$-restricted optimal pebbling configurations. These configurations have been shown in Figure \ref{19}.

    Neopentane, also called $2,2$-dimethylpropane, is a double-branched-chain alkane with five carbon atoms.
    Neopentane is the simplest alkane with a quaternary carbon, and has achiral tetrahedral symmetry. It is one of the three structural isomers with the molecular formula $C_5H_{12}$ (pentanes), the other two being $n$-pentane and isopentane. 
    Here, we denote this alkane with $NP$.  
    The 2-restricted optimal pebbling number of the graph of this alkane is $6$. This alkane has nine $2$-restricted optimal pebbling configurations. These configurations have been shown in Figure \ref{20}.

    \begin{theorem} 
  	\begin{enumerate}
  \item[(i)] $\pi_2^*(IP)=6$ and $P_2^*(IP)=13$. 
  \item[(ii)] $\pi_2^*(NP)=6$ and $P_2^*(NP)=11$. 
  	\end{enumerate}
  \end{theorem} 
  
   \begin{proof}
	\begin{enumerate}
	\item[(i)] The molecule graph $IP $ is constructed by adding a claw with three edges to the leaf of  graph $C_4H_{10} $. 
	
	\item[(ii)]  The molecule graph $NP $ is constructed by adding a claw with three edges to the leaf of  graph $C_4H_{10} $. 
	
	\end{enumerate}
	\qed
    \end{proof}

   \section{Conclusion}

 We investigated the $2$-restricted optimal pebbling number of 
 trees $T$ with distinguishing number two, i.e., $D(T)=2$ and radius at most $2$. We observed that $D(T)\leq \pi_2^*(T)$. We also enumerated the $2$-restricted optimal pebbling configurations of these kind of trees. 
 We studied the $2$-restricted optimal pebbling number of some graphs that are of importance in chemistry such as
 some alkanes. We state some open problems:
 
 \begin{enumerate} 
 	\item What is the $2$-restricted optimal pebbling number of another molecules and dendrimers which are important in Chemistry. 
 	
 	\item What are the bounds for the $2$-restricted optimal pebbling number of a graph based on its distinguishing number. 
 	
 	\end{enumerate}

\end{document}